\documentclass[reqno]{amsart}
\usepackage{amsmath, amsthm, amssymb, amstext}

\usepackage[left=3.5cm,right=3.5cm,top=3cm,bottom=3cm]{geometry}
\usepackage{hyperref,xcolor}
\hypersetup{
 pdfborder={0 0 0},
 colorlinks,
}
\usepackage{enumitem}
\allowdisplaybreaks

\newtheorem{theorem}{Theorem}
\newtheorem{remark}[theorem]{Remark}

\newtheorem{proposition}[theorem]{Proposition}

\newtheorem{definition}[theorem]{Definition}

\DeclareMathOperator*{\divergenz}{div}              %
\DeclareMathOperator*{\ints}{int}         %
\DeclareMathOperator*{\Ss}{S}

\newcommand{\N}{\mathbb{N}}

\newcommand{\R}{\mathbb{R}}

\newcommand{\RN}{\mathbb{R}^N}

\newcommand*\diff{\mathop{}\!\mathrm{d}}

\newcommand{\Lp}[1]{L^{#1}(\Omega)}

\newcommand{\Lprand}[1]{L^{#1}(\partial\Omega)}
\newcommand{\Wp}[1]{W^{1,#1}(\Omega)}

\newcommand{\lan}{\langle}
\newcommand{\ran}{\rangle}
\newcommand{\eps}{\varepsilon}
\newcommand{\ph}{\varphi}
\newcommand{\Om}{\Omega}
\newcommand{\rand}{\partial\Omega}

\newcommand{\into}{\int_{\Omega}}

\newcommand{\intor}{\int_{\partial\Omega}}
\newcommand{\weak}{\rightharpoonup}

\newcommand{\Linf}{L^{\infty}(\Omega)}
\newcommand{\close}{\overline{\Omega}}
\newcommand{\interior}{\ints \left(C^1(\overline{\Omega})_+\right)}

\renewcommand{\l}{\left}
\renewcommand{\r}{\right}

\newcommand{\WH}{W^{1, \mathcal{H}}(\Omega)}
\numberwithin{theorem}{section}
\numberwithin{equation}{section}

\title[Existence results for double phase problems ]{Existence results for double phase problems depending on Robin and Steklov eigenvalues for the $p$-Laplacian}

\author[S. El Manouni]{Said El Manouni}
\address[S. El Manouni]{Imam Mohammad Ibn Saud Islamic University (IMSIU), Faculty of Sciences, Department of Mathematics  P. O. Box 90950, Riyadh 11623, Saudi Arabia}
\email{samanouni@imamu.edu.sa \& manouni@hotmail.com}

\author[G. Marino]{Greta Marino}
\address[G. Marino]{Technische Universit\"{a}t Chemnitz, Fakult\"{a}t f\"{u}r Mathematik, Reichenhainer Stra\ss e 41, 09126 Chemnitz, Germany}
\email{greta.marino@mathematik.tu-chemnitz.de}

\author[P. Winkert]{Patrick Winkert$^{*}$}
\address[P. Winkert]{Technische Universit\"{a}t Berlin, Institut f\"{u}r Mathematik, Stra\ss e des 17.~Juni 136, 10623 Berlin, Germany}
\email{winkert@math.tu-berlin.de}
\thanks{$^{*}$Corresponding author}



\begin{document}

\begin{abstract}
	In this paper we study double phase problems with nonlinear boundary condition and gradient dependence. Under quite general assumptions we prove existence results for such problems where the perturbations satisfy a suitable behavior in the origin and at infinity. Our proofs make use of variational tools, truncation techniques and comparison methods. The obtained solutions depend on the first eigenvalues of the Robin and Steklov eigenvalue problems for the $p$-Laplacian.
\end{abstract}

\maketitle

\section{Introduction}

Let $\Omega\subset \R^N$, $N>1$, be a bounded domain with Lipschitz boundary $\partial\Omega$. We consider the following double phase problem with nonlinear boundary condition and convection term given by
\begin{equation}\label{problem}
	\begin{aligned}
		-\divergenz\l(|\nabla u|^{p-2} \nabla u+ \mu(x) |\nabla u|^{q-2} \nabla u\r)&= h_1(x, u, \nabla u)\quad && \text{in } \Omega, \\
		\l(|\nabla u|^{p-2} \nabla u+ \mu(x) |\nabla u|^{q-2} \nabla u\r) \cdot \nu&= h_2(x, u) && \text{on } \partial\Omega,
	\end{aligned}
\end{equation}
where $\nu(x)$ is the outer unit normal of $\Omega$ at the point $x \in \partial\Omega$, $1< p< q<N$, $0\leq \mu(\cdot) \in \Lp{1}$ and $h_1\colon \Omega \times \R \times \R^N \to \R$ as well as $h_2\colon \partial\Omega \times \R \to \R$ are Carath\'eodory functions which satisfy suitable structure conditions and behaviors near the origin and at infinity, see Sections \ref{section_3} and \ref{section_4} for the precise assumptions. 

The differential operator that appears in \eqref{problem} is the so-called double phase operator which is defined by
\begin{align}\label{operator_double_phase}
	-\divergenz \l(|\nabla u|^{p-2} \nabla u+ \mu(x) |\nabla u|^{q-2} \nabla u\r)\quad \text{for }u\in \Wp{\mathcal{H}}
\end{align}
with an appropriate Musielak-Orlicz Sobolev space $\Wp{\mathcal{H}}$, see its definition in Section \ref{section_2}. Note that when $\inf_{\close} \mu>0$ or $\mu\equiv 0$ then the operator becomes the weighted $(q,p)$-Laplacian or the $p$-Laplacian, respectively. The energy functional $J\colon\WH\to \R$ related to the double phase operator \eqref{operator_double_phase} is given by
\begin{align}\label{integral_minimizer}
	J(u)=\int_\Omega \left(|\nabla  u|^p+\mu(x)|\nabla  u|^q\right)\diff x,
\end{align}
where the integrand has unbalanced growth. The main characteristic of the functional $J$ is the change of ellipticity on the set where the weight function is zero, that is, on the set $\{x\in \Omega: \mu(x)=0\}$. Precisely, the energy density of $J$ exhibits ellipticity in the gradient of order $q$ on the points $x$ where $\mu(x)$ is positive and of order $p$ on the points $x$ where $\mu(x)$ vanishes. 

The first who introduced and studied functionals whose integrands change their ellipticity according to a point was Zhikov \cite{Zhikov-1986} (see also the monograph of Zhikov-Kozlov-Oleinik \cite{Zhikov-Kozlov-Oleinik-1994}) in order to provide models for strongly anisotropic materials. Functionals stated in \eqref{integral_minimizer} have been intensively studied in the past decade concerning regularity for isotropic and anisotropic functionals. We mention the papers of Baroni-Colombo-Mingione \cite{Baroni-Colombo-Mingione-2015,Baroni-Colombo-Mingione-2016,Baroni-Colombo-Mingione-2018}, Baroni-Kuusi-Mingione \cite{Baroni-Kuusi-Mingione-2015}, Byun-Oh \cite{Byun-Oh-2020}, Colombo-Mingione \cite{Colombo-Mingione-2015a,Colombo-Mingione-2015b},
Marcellini \cite{Marcellini-1989,Marcellini-1991}, Ok \cite{Ok-2018,Ok-2020}, Ragusa-Tachikawa \cite{Ragusa-Tachikawa-2020} and the references therein.

In this paper we are going to study problem \eqref{problem} concerning multiplicity of solutions. In the first part of the paper, see Section \ref{section_3}, we prove the existence of a nontrivial weak solution when the function $h_1$ depends on the gradient of the solution. Hence, no variational tools like critical point theory are available. We will make use of the surjectivity result for pseudomonotone operators where in the proof the first eigenvalues of the Robin and Steklov eigenvalue problems for the $p$-Laplacian play an important role. In the second part of the paper we will skip the gradient dependence and prove the existence of two constant sign solutions, one is nonnegative and the other one is nonpositive. Here, we need some stronger conditions on the nonlinearities, for example superlinearity at $\pm\infty$. Again, the solutions depend on the first Robin and Steklov eigenvalues, respectively. We will see that the Steklov eigenvalue problem is the more natural one for problems with nonlinear boundary condition than the Robin eigenvalue problem. 

There are only few works dealing with double phase operators along with a nonlinear boundary condition. Papageorgiou-Vetro-Vetro \cite{Papageorgiou-Vetro-Vetro-2020} studied the Robin problem 
\begin{equation}\label{problem5}
	\begin{aligned}
		-\divergenz\l(a(z)|\nabla u|^{p-2} \nabla u)\r)-\Delta_qu +\xi(z)|u|^{p-2}u&= \lambda f(z,u(z)) \quad && \text{in } \Omega, \\
		\frac{\partial u}{\partial n_\theta} +\beta |u|^{p-2}u&= 0 && \text{on } \partial\Omega,
	\end{aligned}
\end{equation}
where $1<q<p<N$, $\xi\in \Linf$ is a positive potential, $a(z)>0$ for a.\,a.\,$z\in\Omega$ and
\begin{align*}
	\frac{\partial u}{\partial n_\theta}=[a(z)|\nabla u|^{p-2}+|\nabla u|^{q-2}]\frac{\partial u}{\partial n}
\end{align*} 
with $n(\cdot)$ being the outward unit normal on $\partial\Omega$.
Under different assumptions it is shown that problem \eqref{problem5} admits two nontrivial solutions $u_\lambda, \hat{u}_\lambda \in \WH$ for small $\lambda>0$ such that $\|u_\lambda\|_{1,\mathcal{H}}\to +\infty$ and $\|\hat{u}_\lambda\|_{1,\mathcal{H}}\to 0$ as $\lambda \to 0^+$. In Papageorgiou-R\u{a}dulescu-Repov\v{s} \cite{Papageorgiou-Radulescu-Repovs-2020d} the authors proved the existence of multiple solutions in the superlinear and the resonant case for the problem
\begin{equation*}
\begin{aligned}
-\divergenz\l(a_0(z)|\nabla u|^{p-2} \nabla u)\r)-\Delta_qu +\xi(z)|u|^{p-2}u&= f(z,u(z)) \quad && \text{in } \Omega, \\
\frac{\partial u}{\partial n_\theta} +\beta |u|^{p-2}u&= 0 && \text{on } \partial\Omega,
\end{aligned}
\end{equation*}
where $1<q<p\leq N$ and with a positive Lipschitz function $a_0(\cdot)$. Note that our assumptions and our treatment differ from the ones in \cite{Papageorgiou-Radulescu-Repovs-2020d} and \cite{Papageorgiou-Vetro-Vetro-2020}. Also, we allow that the weight function could be zero at some points. Recently, Gasi\'nski-Winkert \cite{Gasinski-Winkert-2021} considered the problem
\begin{equation}\label{problem7}
	\begin{aligned}
		-\divergenz\left(|\nabla u|^{p-2}\nabla u+\mu(x) |\nabla u|^{q-2}\nabla u\right) & =f(x,u)-|u|^{p-2}u-\mu(x)|u|^{q-2}u && \text{in } \Omega,\\
		\left(|\nabla u|^{p-2}\nabla u+\mu(x) |\nabla u|^{q-2}\nabla u\right) \cdot \nu & = g(x,u) &&\text{on } \partial \Omega.
	\end{aligned}
\end{equation} 
Based on the Nehari manifold method it is shown that problem \eqref{problem7} has at least three nontrivial solutions. We point out that the proof for the constant sign solutions in \cite{Gasinski-Winkert-2021} is based on a mountain-pass type argument and so different from the treatment we used in this paper. Very recently, Farkas-Fiscella-Winkert \cite{Farkas-Fiscella-Winkert-2021} studied singular Finsler double phase problems with nonlinear boundary condition and critical growth of the form
\begin{equation}\label{problem8}
	\begin{aligned}
		-\divergenz (A(u)) +u^{p-1}+\mu(x)u^{q-1}& = u^{p^{*}-1}+\lambda \l(u^{\gamma-1}+ g_1(x,u)\r) \quad && \text{in } \Omega,\\
		A(u)\cdot \nu & = u^{p_*-1}+g_2(x,u) &&\text{on } \partial \Omega,\\
		u & > 0 &&\text{in } \Omega,
	\end{aligned}
\end{equation}
where
\begin{align*}
	\divergenz (A(u)):=\divergenz\big(F^{p-1}(\nabla u)\nabla F(\nabla u) +\mu(x) F^{q-1}(\nabla u)\nabla F(\nabla u)\big)
\end{align*}
is the so-called Finsler double phase operator and $(\R^N,F)$ stands for a Minkowski space. The existence of one weak solution of \eqref{problem8} is proven by applying variational tools and truncation techniques.

For existence results for double phase problems with homogeneous Dirichlet boundary condition we refer to the papers of Colasuonno-Squassina \cite{Colasuonno-Squassina-2016} (eigenvalue problem for the double phase operator), Farkas-Winkert \cite{Farkas-Winkert-2020} (Finsler double phase problems), Gasi\'nski-Papa\-georgiou \cite{Gasinski-Papageorgiou-2019} (locally Lipschitz right-hand side), Gasi\'nski-Winkert \cite{Gasinski-Winkert-2020a,Gasinski-Winkert-2020b} (convection and superlinear problems), Liu-Dai \cite{Liu-Dai-2018} (Nehari manifold approach), Marino-Winkert \cite{Marino-Winkert-2020} (systems of double phase operators), Perera-Squassina \cite{Perera-Squassina-2019} (Morse theoretical approach), Zeng-Bai-Gasi\'nski-Winkert \cite{Zeng-Bai-Gasinski-Winkert-2020, Zeng-Gasinski-Winkert-Bai-2020} (multivalued obstacle problems) and the references therein. Related works dealing with certain types of double phase problems can be found in the works of Bahrouni-R\u{a}dulescu-Winkert \cite{Bahrouni-Radulescu-Winkert-2020} (Baouendi-Grushin operator), Barletta-Tornatore \cite{Barletta-Tornatore-2021} (convection problems in Orlicz spaces), Liu-Dai \cite{Liu-Dai-2020} (unbounded domains), Papageorgiou-R\u{a}dulescu-Repov\v{s} \cite{Papageorgiou-Radulescu-Repovs-2019} (discontinuity property for the spectrum), R\u{a}dulescu \cite{Radulescu-2019} (overview about isotropic and anisotropic double phase problems) and Zeng-Bai-Gasi\'nski-Winkert \cite{Zeng-Bai-Gasinski-Winkert-2020b} (convergence properties for double phase problems). Finally, we mention the nice overview article of Mingione-R\u{a}dulescu \cite{Mingione-Radulescu-2021}  about recent developments for problems with nonstandard growth and nonuniform ellipticity.

The paper is organized as follows. In Section \ref{section_2} we recall the main properties of the double phase operator including the properties of the Musielak-Orlicz Sobolev space $\Wp{\mathcal{H}}$. In Section \ref{section_3} we prove the existence of at least one solution of \eqref{problem} when $h_1$ depends on the gradient of the solution, see Theorem \ref{thm1}. The proof is based on the surjectivity result for pseudomonotone operators and on the properties of the eigenvalues of the Robin and Steklov eigenvalue problems for the $p$-Laplacian. Finally, in the last section, we skip the convection term and use variational tools in order to prove the existence of two constant sign solutions for superlinear problems. We consider two different problems. The first problem is treated by properties of the first Steklov eigenvalue and the second one by the first Robin eigenvalue, see Theorems \ref{thm2} and \ref{thm3}. 

\section{Preliminaries}\label{section_2}

In this section we recall some definitions and present the main tools which will be needed in the sequel.

For every $1 \le r< \infty$ we denote by $L^r(\Om)$ and $L^r(\Om; \RN)$ the usual Lebesgue spaces equipped with the norm $\|\cdot \|_r$ and for $1< r< \infty$ we consider the corresponding Sobolev space $W^{1, r}(\Om)$  endowed with the norm $\|\cdot \|_{1,r}$. It is known that $W^{1,r}(\Om) \hookrightarrow L^{\hat{r}}(\Omega)$ is compact for $\hat{r}<r^*$ and continuous for $\hat{r}=r^*$, where $r^*$ is the critical exponent of $r$ defined by
\begin{align}\label{critical_exponent_domain}
	r^*=
	\begin{cases}
		\frac{Nr}{N-r} &\text{if }r<N,\\
		\text{any }\ell\in(r,\infty) & \text{if }r\geq N.
	\end{cases}
\end{align}

On the boundary $\partial \Omega$ of $\Omega$ we consider the $(N-1)$-dimensional Hausdorff (surface) measure $\sigma$ and denote by $\Lprand{r}$ the boundary Lebesgue space with norm $\|\cdot\|_{r,\partial\Omega}$. From the definition of the trace mapping we know that $W^{1,r}(\Omega) \hookrightarrow \Lprand{\tilde{r}}$ is compact for $\tilde{r}<r_*$ and continuous for $\tilde{r}=r_*$, where $r_*$ is the critical exponent of $r$ on the boundary given by
\begin{align}\label{critical_exponent_boundary}
	r_*=
	\begin{cases}
		\frac{(N-1)r}{N-r} &\text{if }r<N,\\
		\text{any }\ell\in(r,\infty) & \text{if }r\geq N.
	\end{cases}
\end{align}
For simplification we will avoid the notation of the trace operator throughout the paper.

In the entire paper we will assume that
\begin{align}\label{condition_poincare}
	1<p<q<N \quad \text{and}\quad
	0\leq \mu(\cdot) \in \Lp{1}.
\end{align}
Note that the conditions in \eqref{condition_poincare} are quite general. In all the other mentioned works for Neumann double phase problems (see, for example, \cite{Farkas-Fiscella-Winkert-2021}, \cite{Gasinski-Winkert-2021}, \cite{Papageorgiou-Radulescu-Repovs-2020d}, \cite{Papageorgiou-Vetro-Vetro-2020})  the condition
\begin{align*}
	\frac{Nq}{N+q-1}<p
\end{align*}
is needed, which is equivalent to $q<p_*$ and so $q<p^*$ is also satisfied. We do not need this restriction in the current paper.

Let $\mathcal{H}\colon \Omega \times [0,\infty)\to [0,\infty)$ be the function defined by
\begin{align*}
	\mathcal H(x,t)= t^p+\mu(x)t^q.
\end{align*}
Based on this we can introduce the modular function given by
\begin{align*}
	\rho_{\mathcal{H}}(u):=\into \mathcal{H}(x,|u|)\diff x=\into \big(|u|^{p}+\mu(x)|u|^q\big)\diff x.
\end{align*}
Then, the Musielak-Orlicz space $L^\mathcal{H}(\Omega)$ is defined by
\begin{align*}
	L^\mathcal{H}(\Omega)=\left \{u ~ \Big | ~ u\colon \Omega \to \R \text{ is measurable and } \rho_{\mathcal{H}}(u)<+\infty \right \}
\end{align*}
equipped with the Luxemburg norm
\begin{align*}
	\|u\|_{\mathcal{H}} = \inf \left \{ \tau >0 : \rho_{\mathcal{H}}\left(\frac{u}{\tau}\right) \leq 1  \right \}.
\end{align*}
From Colasuonno-Squassina \cite[Proposition 2.14]{Colasuonno-Squassina-2016} we know that the space $L^\mathcal{H}(\Omega)$ is a reflexive Banach space. Moreover, we need the seminormed space
\begin{align*}
	L^q_\mu(\Omega)=\left \{u ~ \Big | ~ u\colon \Omega \to \R \text{ is measurable and } \into \mu(x) |u|^q \diff x< +\infty \right \},
\end{align*}
which is endowed with the seminorm
\begin{align*}
	\|u\|_{q,\mu} = \left(\into \mu(x) |u|^q \diff x \right)^{\frac{1}{q}}.
\end{align*}
Analogously, we define $L^q_\mu(\Omega;\R^N)$. 

The Musielak-Orlicz Sobolev space $W^{1,\mathcal{H}}(\Omega)$ is defined by
\begin{align*}
	W^{1,\mathcal{H}}(\Omega)= \left \{u \in L^\mathcal{H}(\Omega) \,:\, |\nabla u| \in L^{\mathcal{H}}(\Omega) \right\}
\end{align*}
equipped with the norm
\begin{align*}
\|u\|_{1,\mathcal{H}}= \|\nabla u \|_{\mathcal{H}}+\|u\|_{\mathcal{H}},
\end{align*}
where $\|\nabla u\|_\mathcal{H}=\|\,|\nabla u|\,\|_{\mathcal{H}}$. As before, we know that $W^{1,\mathcal{H}}(\Omega)$ is a reflexive Banach space. 

The following proposition states the main embedding results for the spaces $\Lp{\mathcal{H}}$ and $\Wp{\mathcal{H}}$. We refer to Crespo-Blanco-Gasi\'nski-Harjulehto-Winkert \cite[Proposition 2.17]{Crespo-Blanco-Gasinski-Winkert-2021}.

\begin{proposition}\label{proposition_embeddings}
	Let \eqref{condition_poincare} be satisfied and let
	\begin{align}\label{critical_exponents}
	p^*:=\frac{Np}{N-p}
	\quad\text{and}\quad
	p_*:=\frac{(N-1)p}{N-p}
	\end{align}
	be the critical exponents to $p$, see \eqref{critical_exponent_domain} and \eqref{critical_exponent_boundary} for $r=p$. Then the following embeddings hold:
	\begin{enumerate}
		\item[\textnormal{(i)}]
		$\Lp{\mathcal{H}} \hookrightarrow \Lp{r}$ and $\Wp{\mathcal{H}}\hookrightarrow \Wp{r}$ are continuous for all $r\in [1,p]$;
		\item[\textnormal{(ii)}]
		$\Wp{\mathcal{H}} \hookrightarrow \Lp{r}$ is continuous for all $r \in [1,p^*]$;
		\item[\textnormal{(iii)}]
		$\Wp{\mathcal{H}} \hookrightarrow \Lp{r}$ is compact for all $r \in [1,p^*)$;
		\item[\textnormal{(iv)}]
		$\Wp{\mathcal{H}} \hookrightarrow \Lprand{r}$ is continuous for all $r \in [1,p_*]$;
		\item[\textnormal{(v)}]
		$\Wp{\mathcal{H}} \hookrightarrow \Lprand{r}$ is compact for all $r \in [1,p_*)$;
		\item[\textnormal{(vi)}]
		$\Lp{\mathcal{H}} \hookrightarrow L^q_\mu(\Omega)$ is continuous.
\end{enumerate}
\end{proposition}

We equip the space $W^{1,\mathcal{H}}(\Omega)$ with the equivalent norm
\begin{align*}
	\|u\|_{0}:=\inf\left\{\lambda >0 :
	\int_\Omega
	\left[\l(\frac{|\nabla u|}{\lambda}\r)^{p}+\mu(x)\l(\frac{|\nabla u|}{\lambda}\r)^{q}+\l(\frac{|u|}{\lambda}\r)^{p}+\mu(x)\l(\frac{|u|}{\lambda}\r)^{q}\,\right]\diff x\le1\right\}.
\end{align*}
For $u \in \Wp{\mathcal{H}}$ let
\begin{align}\label{modular2}
\hat{\rho}_\mathcal{H}(u) =\into \l(|\nabla u|^{p}+\mu(x)|\nabla u|^q\r)\diff x+\into \l(|u|^{p}+\mu(x)|u|^q\r)\diff x.
\end{align}

Based on the proof of Liu-Dai \cite[Proposition 2.1]{Liu-Dai-2018} we have the following relations between the norm $\|\cdot\|_{0}$ and the modular function $\hat{\rho}_\mathcal{H}$, see also Crespo-Blanco-Gasi\'nski-Harjulehto-Winkert \cite[Proposition 2.16]{Crespo-Blanco-Gasinski-Winkert-2021}.

\begin{proposition}\label{proposition_modular_properties2}
	Let \eqref{condition_poincare} be satisfied, let $y\in \WH$ and let $\hat{\rho}_{\mathcal{H}}$ be defined as in  \eqref{modular2}.
	\begin{enumerate}
		\item[\textnormal{(i)}]
		If $y\neq 0$, then $\|y\|_{0}=\lambda$ if and only if $ \hat{\rho}_{\mathcal{H}}(\frac{y}{\lambda})=1$;
		\item[\textnormal{(ii)}]
		$\|y\|_{0}<1$ (resp.\,$>1$, $=1$) if and only if $ \hat{\rho}_{\mathcal{H}}(y)<1$ (resp.\,$>1$, $=1$);
		\item[\textnormal{(iii)}]
		If $\|y\|_{0}<1$, then $\|y\|_{0}^q\leqslant \hat{\rho}_{\mathcal{H}}(y)\leqslant\|y\|_{0}^p$;
		\item[\textnormal{(iv)}]
		If $\|y\|_{0}>1$, then $\|y\|_{0}^p\leqslant \hat{\rho}_{\mathcal{H}}(y)\leqslant\|y\|_{0}^q$;
		\item[\textnormal{(v)}]
		$\|y\|_{0}\to 0$ if and only if $ \hat{\rho}_{\mathcal{H}}(y)\to 0$;
		\item[\textnormal{(vi)}]
		$\|y\|_{0}\to +\infty$ if and only if $ \hat{\rho}_{\mathcal{H}}(y)\to +\infty$.
	\end{enumerate}
\end{proposition}

Let us recall some definitions which we will need in the next sections.

\begin{definition}\label{SplusPM}
	Let $(X,\|\cdot\|_X)$ be a reflexive Banach space, $X^*$ its dual space and denote by $\langle \cdot \,, \cdot\rangle$ its duality pairing. Let $A\colon X\to X^*$, then $A$ is called
	\begin{enumerate}[leftmargin=1cm]
		\item[\textnormal{(i)}]
		to satisfy the $(\Ss_+)$-property if $u_n \weak u$ in $X$ and $\limsup_{n\to \infty} \langle Au_n,u_n-u\rangle \leq 0$ imply $u_n\to u$ in $X$;
		\item[\textnormal{(ii)}]
		pseudomonotone if $u_n \weak u$ in $X$ and $\limsup_{n\to \infty} \langle Au_n,u_n-u\rangle \leq 0$ imply $Au_n \weak Au$ and $\langle Au_n,u_n\rangle \to \langle Au,u\rangle$;
		\item[\textnormal{(iii)}]
		coercive if
		\begin{align*}
		\lim_{\|u\|_X \to \infty} \frac{\langle Au, u \rangle}{\|u\|_X}= \infty.
		\end{align*}
	\end{enumerate}
\end{definition}

\begin{remark}
	The classical definition of pseudomonotonicity is the following one: From $u_n \weak u$ in $X$ and $\limsup_{n\to \infty} \langle Au_n,u_n-u\rangle \leq 0$ we have 
	\begin{align*}
		\liminf_{n \to \infty} \langle Au_n,u_n-v\rangle \geq \langle Au , u-v\rangle \quad\text{for all }v \in X.
	\end{align*}
	This definition is equivalent to the one in Definition \ref{SplusPM}\textnormal{(ii)} when the operator is bounded. Since we are only considering bounded operators, we will use the one in Definition \ref{SplusPM}\textnormal{(ii)}.
\end{remark}
The following surjectivity result for pseudomonotone operators will be used in Section \ref{section_3}. It can be found, for example, in Papageorgiou-Winkert \cite[Theorem 6.1.57]{Papageorgiou-Winkert-2018}.

\begin{theorem}\label{theorem_pseudomonotone}
    Let $X$ be a real, reflexive Banach space, let $A\colon X\to X^*$ be a pseudomonotone, bounded, and coercive operator, and let $b\in X^*$. Then, a solution to the equation $Au=b$ exists.
\end{theorem}

Let $A\colon \Wp{\mathcal{H}}\to \Wp{\mathcal{H}}^*$ be the nonlinear map defined by
\begin{align}\label{operator_representation}
	\begin{split}
		\langle A(u),\ph\rangle_{\mathcal{H}} &=\into \left(|\nabla u|^{p-2}\nabla u+\mu(x)|\nabla u|^{q-2}\nabla u \right)\cdot\nabla\ph \diff x\\
	& \quad +\into \left(|u|^{p-2} u+\mu(x)| u|^{q-2} u \right)\ph \diff x
	\end{split}
\end{align}
for all $u,\ph\in\Wp{\mathcal{H}}$, where $\lan\,\cdot\,,\,\cdot\,\ran_{\mathcal{H}}$ is the duality pairing between $\Wp{\mathcal{H}}$ and its dual space $\Wp{\mathcal{H}}^*$.  The operator $A\colon \Wp{\mathcal{H}}\to \Wp{\mathcal{H}}^*$ has the following properties, see Crespo-Blanco-Gasi\'nski-Harjulehto-Winkert \cite[Proposition 3.5]{Crespo-Blanco-Gasinski-Winkert-2021}.

\begin{proposition}\label{properties_operator_double_phase}
	Let \eqref{condition_poincare} be satisfied. Then, the operator $A$ defined by \eqref{operator_representation} is bounded (that is, it maps bounded sets into bounded sets), continuous, strictly monotone (hence maximal monotone) and it is of type $(\Ss_+)$.
\end{proposition}

For $s \in \R$, we set $s^{\pm}=\max\{\pm s,0\}$ and for $u \in \Wp{\mathcal{H}}$ we define $u^{\pm}(\cdot)=u(\cdot)^{\pm}$. We have
\begin{align*}
	u^{\pm} \in \Wp{\mathcal{H}}, \quad |u|=u^++u^-, \quad u=u^+-u^-.
\end{align*}
For $r>1$ we write $r'=\frac{r}{r-1}$.

Further, we denote by $C^1(\overline{\Omega})_+$ the positive cone
\begin{align*}
	C^1(\overline{\Omega})_+=\left\{u \in C^1(\overline{\Omega}): u(x) \geq 0 \,\,\text{for all }x \in \overline{\Omega}\right\}
\end{align*}
of the ordered Banach space $C^1(\overline{\Omega})$. This cone has a nonempty interior given by
\begin{align*}
	\interior=\left\{u \in C^1(\overline{\Omega}): u(x)>0 \,\,\text{for all }x \in  \close\right\}.
\end{align*}

Let us now recall some basic facts about the spectrum of the negative $r$-Laplacian with Robin and Steklov boundary condition, respectively, for $1<r<\infty$. We refer to the paper of L{\^e} \cite{Le-2006}. The $r$-Laplacian eigenvalue problem with Robin boundary condition is given by
\begin{equation}\label{robin}
	\begin{aligned}
		-\Delta_r u&= \lambda |u|^{r-2} u \quad && \text{in } \Om, \\
		|\nabla u|^{r-2} \nabla u \cdot \nu&= -\beta |u|^{r-2} u && \text{on } \rand,
	\end{aligned}
\end{equation}
where $\beta>0$. We know that problem \eqref{robin} has a smallest eigenvalue $\lambda_{1,r,\beta}^{\text{R}}>0$ which is isolated, simple and it can be variationally characterized by 
\begin{equation}\label{robin-first-characterization}
	\begin{split}
	\lambda_{1,r,\beta}^{\text{R}} =\inf_{u \in \Wp{r}\setminus\{0\}} \frac{\into |\nabla u|^r \diff x+ \beta \intor |u|^r \diff\sigma}{\into |u|^r \diff x}.
	\end{split}
	\end{equation}
By $u_{1,r,\beta}^{\text{R}}$ we denote the normalized (that is, $\|u_{1,r,\beta}^{\text{R}}\|_r=1$) positive eigenfunction corresponding to $\lambda_{1,r,\beta}^{\text{R}}$. We know that $u_{1,r,\beta}^{\text{R}}\in\interior$.

Further, we recall the $r$-Laplacian eigenvalue problem with Steklov boundary condition which is given by
\begin{equation}\label{steklov}
	\begin{aligned}
		-\Delta_r u&= -|u|^{r-2} u \quad && \text{in } \Om, \\
		|\nabla u|^{r-2} \nabla u \cdot \nu & = \lambda |u|^{r-2} u && \text{on } \rand.
	\end{aligned}
\end{equation}
As before, problem \eqref{steklov} has a smallest eigenvalue $\lambda_{1,r}^{\text{S}}>0$ which is isolated, simple and which can be characterized by 
\begin{equation}\label{steklov-first-characterization}
	\begin{split}
		\lambda_{1,r}^{\text{S}} =\inf_{u \in \Wp{r}\setminus\{0\}} \frac{\into |\nabla u|^r \diff x+ \into |u|^r \diff x}{\int_{\partial\Omega} |u|^r \diff \sigma}.
	\end{split}
\end{equation}
The first eigenfunction associated to the first eigenvalue $\lambda_{1,r}^{\text{S}}$ will be denoted by $u_{1,r}^{\text{S}}$ and we can assume it is normalized, that is, $\|u_{1,r}^{\text{S}}\|_{r,\partial\Omega}=1$. We have $u_{1,r}^{\text{S}}\in\interior$. 

\section{Existence results in case of convection}\label{section_3}
In this section we are interested in the existence of a solution of problem \eqref{problem} depending on the first eigenvalues of the Robin and Steklov eigenvalue problems of the $p$-Laplacian. We choose
\begin{align*}
	&&h_1(x,s,\xi)&=f(x,s,\xi)-|s|^{p-2}s-\mu(x)|s|^{q-2}s&&\text{for a.\,a.\,}x\in\Omega,\\
	&&h_2(x,s)&=g(x,s)-\zeta|s|^{p-2}s&&\text{for a.\,a.\,}x\in\partial\Omega,
\end{align*}
for all $s\in\R$ and for all $\xi\in\R^N$ with $\zeta>0$ specified later and Carath\'eodory functions $f$ and $g$ characterized in hypotheses (H1) below. Then \eqref{problem} becomes
\begin{equation}\label{problem1}
	\begin{aligned}
		-\divergenz(|\nabla u|^{p-2} \nabla u+ \mu(x) |\nabla u|^{q-2} \nabla u)&= f(x, u, \nabla u)- |u|^{p-2}u- \mu(x) |u|^{q-2}u  && \text{in } \Om, \\
		(|\nabla u|^{p-2} \nabla u+ \mu(x) |\nabla u|^{q-2} \nabla u) \cdot \nu&= g(x, u) -\zeta|u|^{p-2}u&& \text{on } \rand,
\end{aligned}
\end{equation}
where we assume the following hypotheses:

\begin{enumerate}
	\item[\textnormal{(H1)}]
		The mappings $f\colon \Omega \times \R \times \R^N \to \R $ and $g\colon \partial\Omega \times \R \to \R $ are Carath\'eodory functions with $f(x,0,0)\neq 0$ for a.\,a.\,$x\in\Omega$ such that the following conditions are satisfied:
		\begin{enumerate}
			\item[\textnormal{(i)}]
				There exist $\alpha_1 \in L^{\frac{r_1}{r_1-1}}(\Omega)$, $\alpha_2 \in L^{\frac{r_2}{r_2-1}}(\partial \Omega)$ and $a_1, a_2, a_3 \geq 0$ such that
				\begin{align*}
					&&|f(x, s, \xi)| &\leq a_1 |\xi|^{p \frac{r_1- 1}{r_1}}+ a_2 |s|^{r_1-1}+ \alpha_1(x) &&\text{for a.\,a.\,}x\in\Omega,\\
					&&|g(x, s)| &\leq a_3|s|^{r_2- 1}+\alpha_2(x)&&\text{for a.\,a.\,}x\in\partial \Omega,
				\end{align*}
				for all $ s \in \R$ and for all $\xi \in \R^N$, where $1< r_1< p^*$ and $1< r_2< p_*$ with the critical exponents $p^*$ and $p_*$ stated in \eqref{critical_exponents}.
			\item[\textnormal{(ii)}]
				There exist $w_1 \in L^1(\Omega)$, $w_2 \in L^1(\rand)$ and  $b_1, b_2, b_3 \geq 0$ such that
				\begin{align*}
					&& f(x, s, \xi)s &\le b_1 |\xi|^p+ b_2 |s|^p+ w_1(x) &&\text{for a.\,a.\,}x\in\Omega,\\
					&& g(x, s)s &\leq b_3 |s|^p+ \omega_2 (x)&&\text{for a.\,a.\,}x\in\partial \Omega,
				\end{align*}
				for all $ s \in \R$ and for all $\xi \in \R^N$.
		\end{enumerate} 
\end{enumerate}

A function $u \in \Wp{\mathcal{H}}$ is called a weak solution of problem  \eqref{problem1} if 
\begin{equation}\label{weak1}
	\begin{split}
		& \into \l(|\nabla u|^{p-2} \nabla u+ \mu(x) |\nabla u|^{q-2} \nabla u\r) \cdot \nabla \varphi \diff x+ \into\l(|u|^{p-2}u+ \mu(x) |u|^{q-2}u\r) \varphi \diff x \\
		& = \into f(x, u, \nabla u) \varphi \diff x+ \int_{\partial\Omega} g(x, u) \varphi \diff \sigma-\zeta \int_{\partial\Omega}|u|^{p-2}u\ph\diff \sigma 
	\end{split}
	\end{equation}
is satisfied for all $\varphi \in \Wp{\mathcal H}$. It is clear that this definition is well-defined.

The main result in this section is the following one.

\begin{theorem}\label{thm1}
        Let hypotheses  \eqref{condition_poincare} and \textnormal{(H1)} be satisfied. Then, there exists a nontrivial weak solution $ \hat{u}  \in \Wp{\mathcal H}\cap \Linf$ of problem \eqref{problem1} provided one of the following assertions is satisfied:
        \begin{enumerate}
        	\item[\textnormal{(A)}]
        		$b_1+b_2\l(\lambda_{1,p,\beta}^{\text{R}}\r)^{-1}<1 \ $  and $\ b_2\beta \l(\lambda_{1,p,\beta}^{\text{R}}\r)^{-1}+b_3<\zeta$;\\
        	\item[\textnormal{(B)}]
        		$\max\{b_1,b_2\}+b_3 \l(\lambda_{1,p}^{\text{S}}\r)^{-1}<1$ and $\zeta \geq 0$.
        \end{enumerate}
    Here $\lambda_{1,p,\beta}^{\text{R}}$ is the first eigenvalue of the $p$-Laplacian with Robin boundary condition with $\beta>0$ and $\lambda_{1,p}^{\text{S}}$ stands for the first eigenvalue of the $p$-Laplacian with Steklov boundary condition, see \eqref{robin} and \eqref{steklov}, respectively.
\end{theorem}

\begin{proof}
	Let $ \tilde{N}_{f}\colon \Wp{\mathcal{H}}  \subset \Lp{r_1}  \to \Lp{r_1'} $ and $\tilde{N}_g \colon L^{r_2}(\rand) \to L^{r_2'}(\rand)$ be the Nemytskij operators corresponding to the functions $f\colon \Omega \times \R \times \R^N \to \R $ and $g\colon \partial\Omega \times \R \to \R $, respectively. Furthermore, we denote by $i^*\colon L^{r_1'}(\Om)  \to W^{1, \mathcal H}(\Om)^* $ the adjoint operator of the embedding $ i\colon W^{1, \mathcal H}(\Om) \to \Lp{r_1} $ and $ j^* \colon L^{r_2'}(\rand) \to \WH^*$
	stands for the adjoint operator of the embedding $j\colon \WH \to L^{r_2}(\rand)$. Then we define
    \begin{align*}
        N_{f}:= i^* \circ \tilde{N}_{f}&\colon W^{1, \mathcal H}(\Om)  \to W^{1, \mathcal H}(\Om)^*,\\
        N_g:= j^* \circ \tilde{N}_g\circ j &\colon \WH \to \WH^*,
    \end{align*}
    which are both bounded and continuous operators due to hypothesis (H1)(i). Moreover, we define $N_\zeta\colon\WH\to \WH^*$ by
	\begin{align*}
		N_\zeta&:=i_\zeta^*\circ \left(\zeta |\cdot|^{p-2}\cdot\right)\circ i_\zeta,
	\end{align*}
	where $i_\zeta^*\colon L^{p'}(\Omega)  \to W^{1, \mathcal H}(\Om)^*$ is the adjoint operator of the embedding $ i_\zeta\colon W^{1, \mathcal H}(\Om) \to \Lp{p} $.
	 
	Now we can define the operator $\mathcal{A}\colon\WH\to\WH^*$ given by 
	\begin{align*}
		\mathcal{A} (u):= A(u)- N_{f}(u)- N_{g}(u)+N_\zeta(u).
    \end{align*}
    Taking the growth conditions in \textnormal{(H1)(i)} into account, it is clear that $\mathcal{A}\colon\WH\to\WH^*$ maps bounded sets into bounded sets. In order to show the pseudomonotonicity, let $\{u_n\}_{n \in \N} \subset \WH$ be such that 
	\begin{equation} \label{weak_lim}
		u_n \rightharpoonup u \quad \text{in }W^{1, \mathcal H}(\Om) \quad\text{and}\quad \limsup_{n \to \infty} \langle \mathcal A u_n, u_n- u \rangle_{\mathcal H} \le 0.
	\end{equation}
	From the compact embeddings $\WH \hookrightarrow \Lp{\hat{r}}$ for any $\hat{r}<p^*$ and $\WH \hookrightarrow \Lprand{\tilde{r}}$ for any $\tilde{r}<p_*$, see Proposition \ref{proposition_embeddings}\textnormal{(iii)} and \textnormal{(v)}, along with \eqref{weak_lim} we have
	\begin{equation*}
		u_n \to u \quad \text{in } \Lp{r_1} \quad \text{and} \quad u_n \to u \quad \text{in } \Lprand{r_2}, \Lprand{p}.
	\end{equation*}
	Applying the growth conditions in \textnormal{(H1)(i)} along with H\"older's inequality gives
	\begin{align*}
		\begin{split}
			& \into f(x, u_n, \nabla u_n)(u_n- u)\diff x\\
			&\leq a_1 \into |\nabla u_n|^{p \frac{r_1-1}{r_1}} |u_n-u|\diff x+ a_2 \into |u_n|^{r_1-1}|u_n- u| \diff x + \into |\alpha_1(x)|\,|u_n-u| \diff x \\
			& \le a_1 \|\nabla u_n\|_p^{p\frac{r_1-1}{r_1}} \|u_n-u\|_{r_1} + a_2 \|u_n\|_{r_1}^{r_1-1} \|u_n- u\|_{r_1}+ \|\alpha_1\|_{\frac{r_1}{r_1-1}} \|u_n-u\|_{r_1} \longrightarrow 0
		\end{split}
	\end{align*}
	and
	\begin{align*}
		\begin{split}
		\intor g(x, u_n) (u_n- u) \diff\sigma &\le a_3 \int_{\partial\Omega} |u_n|^{r_2-1} |u_n-u| \diff \sigma+\int_{\partial\Omega} |\alpha_2(x)|\,|u_n-u| \diff \sigma \\
		& \le a_3 \|u_n\|_{r_2,\partial\Omega}^{r_2-1}  \|u_n- u\|_{r_2, \rand} + \|\alpha_2\|_{\frac{r_2}{r_2-1}, \rand} \|u_n- u\|_{r_2, \rand} \longrightarrow 0.
		\end{split}
	\end{align*}
	Furthermore, again by H\"older's inequality, we have
	\begin{align*}
		\begin{split}
			& \zeta\int_{\partial\Omega}|u_n|^{p-2}u_n(u_n-u)\diff \sigma\leq \zeta  \|u_n\|_{p,\partial\Omega}^{p-1} \|u_n- u\|_{p,\partial\Omega}\longrightarrow 0.
		\end{split}
	\end{align*}
	Replacing $u$ by $u_n$ and $\ph$ by $u_n-u$ in the weak formulation in \eqref{weak1} and using the considerations above leads to
	\begin{align}\label{limsup1}
		\limsup_{n \to \infty} \langle A (u_n), u_n-u \rangle_{\mathcal H} 
		= \limsup_{n \to \infty} \langle \mathcal A (u_n), u_n- u \rangle_{\mathcal H} \le 0.
	\end{align}
	From Proposition \ref{properties_operator_double_phase} we know that $A$ fulfills the $(\Ss_+)$-property. Therefore, from \eqref{weak_lim} and \eqref{limsup1} we conclude that
	\begin{align*}
		u_n  \to  u  \quad \text{in } \WH.
	\end{align*}
	Since $\mathcal A$ is continuous we have $\mathcal A(u_n)  \to \mathcal A (u) $ in $\WH^*$ which shows  that $\mathcal{A} $ is pseudomonotone.  

	Let us now prove that $\mathcal{A}\colon\WH\to\WH^*$ is coercive. We distinguish between two cases. 
	
	{\bf Case I:} Condition \textnormal{(A)} is satisfied.
	
	From the $p$-Laplace eigenvalue problem with Robin boundary condition, see \eqref{robin} and \eqref{robin-first-characterization} for $r=p$, we know that
	\begin{align}\label{robin_2}
		\|u\|_p^p \leq \l(\lambda_{1,p,\beta}^{\text{R}}\r)^{-1} \l(\|\nabla u\|_p^p+\beta \|u\|_{p,\partial\Omega}^p \r)\quad \text{for all }u \in \Wp{p}.
	\end{align} 
	Let $u \in \WH$ be such that $\|u\|_{0}>1$ and note that $\WH\subseteq \Wp{p}$. Then, from \textnormal{(H1)(ii)}, \eqref{robin_2}, \textnormal{(A)} and Proposition \ref{proposition_modular_properties2}\textnormal{(iv)} we obtain
	\begin{align*}
			\langle \mathcal{A} (u), u \rangle_{\mathcal H} &= \into \l(|\nabla u|^{p-2} \nabla u+ \mu(x) |\nabla u|^{q-2} \nabla u \r) \cdot \nabla u \diff x+ \into \l(|u|^{p-2}u+ \mu(x) |u|^{q-2}u\r) u \diff x \\
			& \quad - \into f(x, u, \nabla u) u\diff x- \int_{\partial\Omega} g(x, u)u \diff\sigma+\zeta \int_{\partial\Omega} |u|^{p}\diff \sigma \\
			& \ge \|\nabla u\|_p^p + \|\nabla u\|_{q, \mu}^q +\|u\|_p^p+\|u\|_{q, \mu}^q-b_1\|\nabla u\|_p^p-b_2\|u\|_p^p-\|\omega_1\|_1\\
			& \quad -b_3 \|u\|_{p,\partial\Omega}^p- \|\omega_2\|_{1, \rand}+\zeta\|u\|_{p,\partial\Omega}^p \\
			&\geq \l(1-b_1-b_2\l(\lambda_{1,p,\beta}^{\text{R}}\r)^{-1}\r) \l( \|\nabla u\|_p^p+\|u\|_p^p\r)+\|\nabla u\|_{q, \mu}^q+\|u\|_{q, \mu}^q\\
			& \quad +\l(\zeta-b_2 \beta\l(\lambda_{1,p,\beta}^{\text{R}}\r)^{-1} -b_3\r)\|u\|_{p,\partial\Omega}^p-\|\omega_1\|_1- \|\omega_2\|_{1, \rand}\\
			&\geq \l(1-b_1-b_2\l(\lambda_{1,p,\beta}^{\text{R}}\r)^{-1}\r) \l( \|\nabla u\|_p^p+\|u\|_p^p+\|\nabla u\|_{q, \mu}^q+\|u\|_{q, \mu}^q\r)-\|\omega_1\|_1- \|\omega_2\|_{1, \rand}\\
			&= \l(1-b_1-b_2\l(\lambda_{1,p,\beta}^{\text{R}}\r)^{-1}\r) \hat{\rho}_{\mathcal{H}}(u)-\|\omega_1\|_1- \|\omega_2\|_{1, \rand}\\
			&\geq \l(1-b_1-b_2\l(\lambda_{1,p,\beta}^{\text{R}}\r)^{-1}\r) \|u\|_{0}^p-\|\omega_1\|_1- \|\omega_2\|_{1, \rand}.
	\end{align*}
This shows the coercivity of $\mathcal A$. 

{\bf Case II:} Condition \textnormal{(B)} is satisfied.

From the Steklov $p$-Laplace eigenvalue problem, see \eqref{steklov} and \eqref{steklov-first-characterization} for $r=p$, we have the inequality
\begin{align}\label{steklov_2}
	\|u\|_{p,\partial\Omega}^p \leq \l(\lambda_{1,p}^{\text{S}}\r)^{-1} \l(\|\nabla u\|_p^p+ \|u\|_{p}^p \r)\quad \text{for all }u \in \Wp{p}.
\end{align} 
As before, let $u \in \WH$ be such that $\|u\|_{0}>1$ and note again that $\WH\subseteq \Wp{p}$. Applying \textnormal{(H1)(ii)}, \eqref{steklov_2}, \textnormal{(B)} and Proposition \ref{proposition_modular_properties2}\textnormal{(iv)} one gets
\begin{align*}
\langle \mathcal{A} (u), u \rangle_{\mathcal H} 
&= \into \l(|\nabla u|^{p-2} \nabla u+ \mu(x) |\nabla u|^{q-2} \nabla u \r) \cdot \nabla u \diff x+ \into \l(|u|^{p-2}u+ \mu(x) |u|^{q-2}u\r) u \diff x \\
& \quad - \into f(x, u, \nabla u) u\diff x- \int_{\partial\Omega} g(x, u)u \diff\sigma+\zeta \int_{\partial\Omega} |u|^{p}\diff \sigma \\
& \ge \|\nabla u\|_p^p + \|\nabla u\|_{q, \mu}^q +\|u\|_p^p+\|u\|_{q, \mu}^q-b_1\|\nabla u\|_p^p-b_2\|u\|_p^p-\|\omega_1\|_1\\
& \quad -b_3 \|u\|_{p,\partial\Omega}^p- \|\omega_2\|_{1, \rand}+\zeta\|u\|_{p,\partial\Omega}^p \\
&\geq \l(1-\max\{b_1,b_2\}-b_3\l(\lambda_{1,p}^{\text{S}}\r)^{-1}\r) \l( \|\nabla u\|_p^p+\|u\|_p^p\r)+\|\nabla u\|_{q, \mu}^q+\|u\|_{q, \mu}^q\\
&\quad -\|\omega_1\|_1- \|\omega_2\|_{1, \rand}\\
&\geq \l(1-\max\{b_1,b_2\}-b_3\l(\lambda_{1,p}^{\text{S}}\r)^{-1}\r)\hat{\rho}_{\mathcal{H}}(u)-\|\omega_1\|_1- \|\omega_2\|_{1, \rand}\\
&\geq  \l(1-\max\{b_1,b_2\}-b_3\l(\lambda_{1,p}^{\text{S}}\r)^{-1}\r)\|u\|_{0}^p-\|\omega_1\|_1- \|\omega_2\|_{1, \rand}.
\end{align*}
Hence, $\mathcal{A}\colon\WH\to\WH^*$ is again coercive. 

We have shown that $\mathcal{A}\colon\WH\to\WH^*$ is a bounded, pseudomonotone and coercive operator. From Theorem \ref{theorem_pseudomonotone} we find an element $\hat{u}\in\WH$ such that $\mathcal{A} (\hat{u})= 0$ with $\hat{u}\neq 0$ since $f(x,0,0)\neq 0$ for a.\,a.\,$x\in\Omega$. In view of the definition  of  $\mathcal{A}$, we see that $\hat{u}$ turns out to be a nontrivial weak solution of problem \eqref{problem1}. Similar to Theorem 3.1 of Gasi\'nski-Winkert \cite{Gasinski-Winkert-2021} we can show the boundedness of $\hat{u}$. The proof is complete.
\end{proof}

\section{Constant sign solutions for superlinear perturbations}\label{section_4}

In this section we are interested in constant sign solutions for problems of type \eqref{problem} without convection term but with superlinear nonlinearities. We are going to consider the cases of the dependence on Robin and Steklov eigenvalues separately. We start with the Steklov case and set
\begin{align*}
&&h_1(x,s,\xi)&=-\vartheta |s|^{p-2}s-\mu(x)|s|^{q-2}s-f(x,s)&&\text{for a.\,a.\,}x\in\Omega,\\
&&h_2(x,s)&=\zeta|s|^{p-2}s-g(x,s)&&\text{for a.\,a.\,}x\in\partial\Omega,
\end{align*}
for all $s\in\R$, $\vartheta, \zeta>0$ to be specified and Carath\'eodory functions $f$ and $g$ which satisfy hypotheses (H2) below. With this choice, \eqref{problem} can be written as
\begin{equation}\label{problem2}
\begin{aligned}
-\divergenz(|\nabla u|^{p-2} \nabla u+ \mu(x) |\nabla u|^{q-2} \nabla u)&= -\vartheta |u|^{p-2}u-\mu(x)|u|^{q-2}u- f(x,u)  && \text{in } \Om, \\
(|\nabla u|^{p-2} \nabla u+ \mu(x) |\nabla u|^{q-2} \nabla u) \cdot \nu&= \zeta|u|^{p-2}u-g(x,u)&& \text{on } \rand,
\end{aligned}
\end{equation}
where the following conditions are supposed:
\begin{enumerate}
	\item[\textnormal{(H2)}]
	The nonlinearities $f\colon \Omega \times \R  \to \R $ and $g\colon \partial\Omega \times \R \to \R $ are assumed to be Carath\'eodory functions which satisfy the subsequent hypotheses:
	\begin{enumerate}
		\item[\textnormal{(i)}]
			$f$ and $g$ are bounded on bounded sets.
		\item[\textnormal{(ii)}]
		It holds
		\begin{align*}
			&\lim_{s \to \pm \infty}\,\frac{f(x,s)}{|s|^{q-2}s}=+\infty \quad\text{uniformly for a.\,a.\,}x\in\Omega,\\
			&\lim_{s \to \pm \infty}\,\frac{g(x,s)}{|s|^{q-2}s}=+\infty \quad\text{uniformly for a.\,a.\,}x\in\partial\Omega.
		\end{align*}
		\item[\textnormal{(iii)}]
		It holds
		\begin{align*}
			&\lim_{s \to 0}\,\frac{f(x,s)}{|s|^{q-2}s}=0 \quad\text{uniformly for a.\,a.\,}x\in\Omega,\\
			&\lim_{s \to 0}\,\frac{g(x,s)}{|s|^{p-2}s}=0 \quad\text{uniformly for a.\,a.\,}x\in\partial\Omega.
		\end{align*}
	\end{enumerate}
\end{enumerate}

We say that $u \in \WH$ is a weak solution of problem \eqref{problem2} if
\begin{equation*}
	\begin{split}
		& \into \l( |\nabla u|^{p-2} \nabla u+ \mu(x) |\nabla u|^{q-2} \nabla u\r) \cdot \nabla \ph \diff x+\into \l( \vartheta |u|^{p-2} u+ \mu(x) |u|^{q-2} u\r) \ph \diff x\\
		& = \into\l(-f(x,u)\r) \ph\diff x+ \int_{\partial\Omega} \l(\zeta |u|^{p-2} u -g(x,u)\r)\ph\diff \sigma
	\end{split}
\end{equation*}
is fulfilled for all $ \ph \in \WH$. 

The following theorem  states the existence of constant sign solutions where the parameter $\zeta$ depends on the first Steklov eigenvalue for the $p$-Laplacian, namely $\lambda_{1,p}^{\text{S}}$.

\begin{theorem}\label{thm2}
	Let hypotheses  \eqref{condition_poincare} and \textnormal{(H2)} be satisfied. Furthermore, let $\vartheta \in (0,1]$ and let $\zeta>\lambda_{1,p}^{\text{S}}$ with $\lambda_{1,p}^{\text{S}}$ being the first eigenvalue of the Steklov eigenvalue problem of the $p$-Laplacian stated in \eqref{steklov}. Then, problem \eqref{problem2} has at least two nontrivial weak solutions $u_0, v_0\in \WH\cap \Linf$ such that $u_0\geq 0$ and $v_0\leq 0$.
\end{theorem} 

\begin{proof}
	From hypothesis \textnormal{(H2)(ii)} we know that we can find constants $M_1,M_2=M_2(\zeta)>1$ such that
	\begin{align}\label{est_f_g}
		\begin{split}
			f(x,s)s &\geq   |s|^q \quad \text{for a.a. } x \in \Omega \text{ and all }|s| \geq M_1,\\
			g(x,s)s &\geq \zeta |s|^{q} \quad \text{for a.a. } x \in \Omega \text{ and all }|s| \geq M_2.
		\end{split}
	\end{align}
	We set $M_3=\max \left(M_1,M_2\right)$ and take a constant function $\overline{u} \equiv \varsigma \in \left[M_3,+\infty\right)$.  Applying (\ref{est_f_g}), $p<q$ and $M_3>1$ yields
	\begin{align}\label{corollary_A_super}
		0 \geq -f(x,\overline{u}) \quad \text{for a.\,a.\,}x\in\Omega
		\quad \text{and}\quad 
		0 \geq \zeta \overline{u}^{p-1}-g(x,\overline{u}) \quad \text{for a.\,a.\,} x\in\partial \Omega.
	\end{align}
	Analogously, we can choose $\underline{v}\equiv -\varsigma$ in order to get
	\begin{align*}
		0 \leq -f(x,\underline{v}) \quad \text{for a.\,a.\,}x\in\Omega
		\quad \text{and}\quad 
		0 \leq \zeta|\underline{v}|^{p-2}\underline{v}-g(x,\underline{v}) \quad \text{for a.\,a.\,} x\in\partial \Omega.
	\end{align*}

	Now, we introduce the cut-off functions $\theta^{\pm}\colon\Omega\times\R\to\R$ and $\theta^{\pm}_\zeta\colon \partial\Omega\times\R\to\R$ defined by
	\begin{align}\label{truncation_one}
		\begin{split}
			\theta^+(x,s)&=
			\begin{cases}
				0 \qquad & \text{if } s<0\\
				-f(x,s)& \text{if }0 \leq s \leq \overline{u}\\
				-f(x,\overline{u})& \text{if } \overline{u}<s
			\end{cases},\\
			\theta^+_\zeta(x,s)&=
			\begin{cases}
				0 \qquad & \text{if } s<0\\
				\zeta s^{p-1}-g(x,s) & \text{if }0 \leq s \leq \overline{u}\\
				\zeta \overline{u}^{p-1}-g(x,\overline{u}) & \text{if } \overline{u}<s
			\end{cases},\\
			\theta^-(x,s)&=
			\begin{cases}
				-f(x,\underline{v}) \qquad & \text{if } s<\underline{v}\\
				-f(x,s) & \text{if }\underline{v} \leq s \leq 0\\
				0 & \text{if } 0<s
			\end{cases},\\
			\theta^-_\zeta(x,s)&=
			\begin{cases}
				\zeta |\underline{v}|^{p-2}\underline{v}-g(x,\underline{v}) \qquad & \text{if } s<\underline{v}\\
				\zeta |s|^{p-2}s-g(x,s) & \text{if }\underline{v} \leq s \leq 0\\
				0 & \text{if } 0<s
			\end{cases},
		\end{split}
	\end{align}
	which are Carath\'{e}odory functions. We set
	\begin{align*}
		\Theta^{\pm}(x,s)=\int^s_0 \theta^{\pm}(x,t)\diff t
		\quad\text{and}\quad
		\Theta^{\pm}_\zeta(x,s)=\int^s_0 \theta^{\pm}_\zeta(x,t)\diff t.
	\end{align*}
	Now we consider the $C^1$-functionals $\Gamma^{\pm}\colon \WH \to \R$ defined by
	\begin{align*}
	\Gamma^{\pm}(u)
	&=\frac{1}{p}\|\nabla u\|_p^p+\frac{1}{q}\|\nabla u\|_{q,\mu}^q+\frac{\vartheta}{p}\| u\|_p^p+\frac{1}{q}\|u\|_{q,\mu}^q-\into \Theta^{\pm}(x,u) \diff x-\int_{\partial \Omega}\Theta^{\pm}_\zeta(x,u) \diff\sigma.
	\end{align*}
	Furthermore, we write $F(x,s)=\int^s_0 f(x,t)\diff t$ and $G(x,s)=\int^s_0 g(x,t)\diff t$.

	We first investigate the existence of the nonnegative solution. Due to the truncations in \eqref{truncation_one} it is clear that the functional $\Gamma^+$ is coercive and also sequentially weakly lower semicontinuous. Hence, its global minimizer $u_0\in\WH$ exists, that is
	\begin{align*}
		\Gamma^+(u_0)= \inf \left[\Gamma^+(u)\,:\, u \in \WH \right].
	\end{align*}
	From hypotheses \textnormal{(H2)(iii)}, for given  $\varepsilon_1,\varepsilon_2>0$, there exist $\delta_1=\delta_1(\varepsilon_1),\delta_2=\delta_2(\varepsilon_2) \in (0,\overline{u})$ such that
	\begin{align}\label{estimate_FG}
	\begin{split}
	& F(x,s) \leq \frac{\varepsilon_1}{q} |s|^q	 \quad \text{for a.a. }x \in \Omega \text{ and for all } |s|\leq \delta_1,\\
	& G(x,s) \leq \frac{\varepsilon_2}{p} |s|^p \quad \text{for a.a. }x \in \partial \Omega \text{ and for all } |s|\leq \delta_2.
	\end{split}
	\end{align}
	We set $\delta:=\min(\delta_1,\delta_2)$. Recall that $u_{1,p}^{\text{S}}$ is the first eigenfunction corresponding to the first eigenvalue $\lambda_{1,p}^{\text{S}}$ of the eigenvalue problem of the $p$-Laplacian with Steklov boundary condition, see \eqref{steklov}. We may suppose that it is normalized, that is, $\|u_{1,p}^{\text{S}}\|_{p,\partial\Omega}=1$. Since $u_{1,p}^{\text{S}}\in \interior$, we may choose $t \in (0,1)$ small enough such that $t u_{1,p}^{\text{S}}(x) \in [0,\delta]$ for all $x \in \overline{\Omega}$. Because of \eqref{truncation_one}, \eqref{estimate_FG} and $\delta<\overline{u}$ it follows that
	\begin{align}\label{eq_1}
	\begin{split}
	&\Gamma^+\l(t u_{1,p}^{\text{S}}\r)\\
	&=\frac{1}{p}\l\|\nabla \l(t u_{1,p}^{\text{S}}\r)\r\|_p^p+\frac{1}{q}\l\|\nabla \l(t u_{1,p}^{\text{S}}\r)\r\|_{q,\mu}^q+\frac{\vartheta}{p}\l\| t u_{1,p}^{\text{S}}\r\|_p^p+\frac{1}{q}\l\|t u_{1,p,}^{\text{S}}\r\|_{q,\mu}^q\\
	&\quad -\into \Theta^{+}\l(x,t u_{1,p}^{\text{S}}\r) \diff x-\int_{\partial \Omega}\Theta^{+}_\zeta\l(x,t u_{1,p}^{\text{S}}\r) \diff\sigma\\
	&\leq \frac{t^p}{p} \lambda_{1,p}^{\text{S}}+\frac{t^q}{q}\l\|\nabla  u_{1,p}^{\text{S}}\r\|_{q,\mu}^q+\frac{t^q}{q}\l\|u_{1,p}^{\text{S}}\r\|_{q,\mu}^q+\into F\l(x,tu_{1,p}^{\text{S}}\r)\diff x-\frac{\zeta t^p}{p}\\ 
	&\quad +\int_{\partial\Omega} G\l(x,tu_{1,p}^{\text{S}}\r)\diff \sigma\\
	&\leq \frac{t^p}{p} \lambda_{1,p}^{\text{S}}+\frac{t^q}{q}\l\|\nabla  u_{1,p}^{\text{S}}\r\|_{q,\mu}^q+\frac{t^q}{q}\l\|u_{1,p}^{\text{S}}\r\|_{q,\mu}^q+\frac{\eps_1 t^q}{q}\l\| u_{1,p}^{\text{S}}\r\|_q^q-\frac{\zeta t^p}{p} +\frac{\eps_2t^p}{p}\\
	& =t^p\l(\frac{\lambda_{1,p}^{\text{S}}-\zeta+\eps_2}{p}\r)+t^q\l(\frac{\l\|\nabla  u_{1,p}^{\text{S}}\r\|_{q,\mu}^q+\l\|u_{1,p}^{\text{S}}\r\|_{q,\mu}^q+\eps_1 \l\| u_{1,p}^{\text{S}}\r\|_q^q}{q}\r).
	\end{split}
	\end{align}
	By assumption, we know that $\zeta>\lambda_{1,p}^{\text{S}}$. So we may choose $\eps_1,\eps_2>0$ such that
	\begin{align*}
		0<\eps_1<\infty\quad\text{and}\quad 0<\eps_2<\zeta-\lambda_{1,p}^{\text{S}}.
	\end{align*}
	From this choice and since $p<q$ we obtain from \eqref{eq_1}
	\begin{align*}
		\Gamma^+\l(t u_{1,p}^{\text{S}}\r)<0 \quad\text{for all sufficiently small }t>0.
	\end{align*}
	Therefore, we know now that 
	\begin{align*}
		\Gamma^+\l(u_0\r)<0=\Gamma^+\l(0\r).
	\end{align*}
	Hence, $u_0\neq 0$.
	
	Since $u_0$ is a global minimizer of $\Gamma^+$ we have $(\Gamma^+)'(u_0)=0$, that is,
	\begin{align}\label{7}
		\begin{split}
			& \into \left(|\nabla u_0|^{p-2}\nabla u_0+\mu(x)|\nabla u_0|^{q-2}\nabla u_0 \right)\cdot\nabla \ph \diff x\\
			&\quad +\into \left(\vartheta |u_0|^{p-2} u_0+\mu(x)|u_0|^{q-2} u_0 \right) \ph \diff x\\ 
			& =\into \theta^+\l(x,u_0\r)\ph\diff x+\int_{\partial\Omega} \theta_\zeta^+\l(x,u_0\r) \ph\diff \sigma
		\end{split}
	\end{align}
	for all $\ph \in \WH$. First we take $\ph=-u_0^-\in\WH$ as test function in \eqref{7}. We obtain
	\begin{align*}
		\l\|\nabla u_0^-\r\|_p^p+\l\|\nabla u_0^-\r\|_{q,\mu}^q+\l\| u_0^-\r\|_p^p+\l\|u_0^-\r\|_{q,\mu}^q = 0,
	\end{align*}
	which yields $u_0^-=0$ and so $u_0\geq 0$. Second we choose $\ph=\l(u_0-\overline{u}\r)^+ \in \WH$ as test function in \eqref{7} which results in
	\begin{align}\label{7new}
		\begin{split}
			& \into \left(|\nabla u_0|^{p-2}\nabla u_0+\mu(x)|\nabla u_0|^{q-2}\nabla u_0 \right)\cdot\nabla \l(u_0-\overline{u}\r)^+ \diff x\\
			&\quad +\into \left(\vartheta u_0^{p-1}+\mu(x)u_0^{q-1} \right) \l(u_0-\overline{u}\r)^+ \diff x\\
			& =\into \theta^+(x,u_0)\l(u_0-\overline{u}\r)^+ \diff x+\int_{\partial\Omega} \theta_\zeta^+\l(x,u_0\r)\l(u_0-\overline{u}\r)^+\diff \sigma\\
			& =\into (-f(x,\overline{u}))\l(u_0-\overline{u}\r)^+\diff x+\int_{\partial\Omega}\left( \zeta \overline{u}^{p-1}-g(x, \overline{u})\right)\l(u_0-\overline{u}\r)^+ \diff \sigma\\
	& \leq 0,
	\end{split}
	\end{align}
	by \eqref{corollary_A_super}. First note that
	\begin{align}\label{7new1}
		\begin{split}
			&\into \left(|\nabla u_0|^{p-2}\nabla u_0+\mu(x)|\nabla u_0|^{q-2}\nabla u_0 \right)\cdot\nabla \l(u_0-\overline{u}\r)^+ \diff x\\
			&\geq \vartheta \into \left(|\nabla (u_0-\overline{u})^+|^{p}+\mu(x)|\nabla (u_0-\overline{u})^+|^{q} \right) \diff x.
		\end{split}
	\end{align}
	Since $u_0>\overline{u}>1$ on the set $\{u_0>\overline{u}\}$ we have
	\begin{align}\label{7new2}
		\begin{split}
		& \into \left(\vartheta u_0^{p-1}+\mu(x)u_0^{q-1} \right) \l(u_0-\overline{u}\r)^+ \diff x\\
		&\geq \vartheta \int_{\{u_0>\overline{u}\}} \left(u_0^{p-1}+\mu(x)u_0^{q-1} \right) \l(u_0-\overline{u}\r) \diff x\\
		&\geq \vartheta \int_{\{u_0>\overline{u}\}} \left((u_0-\overline{u})^{p-1}+\mu(x)(u_0-\overline{u})^{q-1} \right) \l(u_0-\overline{u}\r) \diff x\\
		&=\vartheta \into \left(((u_0-\overline{u})^+)^{p}+\mu(x)((u_0-\overline{u})^+)^{q} \right)\diff x.
		\end{split}	
	\end{align}
	Combining \eqref{7new} with \eqref{7new1} as well as \eqref{7new2} and using Proposition \ref{proposition_modular_properties2}(\textnormal{iii}), (\textnormal{iv}) implies that
	\begin{align*}
	\begin{split}
		& \vartheta \min \{\|(u_0-\overline{u})^+\|_0^p,\|(u_0-\overline{u})^+\|_0^q \} \leq \vartheta \hat{\rho}_\mathcal{H}((u_0-\overline{u})^+) \leq 0.
	\end{split}
	\end{align*}
	Hence, $u_0\leq \overline{u}$ and so $u_0 \in [0,\overline{u}]$. By the definition of the truncations in \eqref{truncation_one} we see that $u_0 \in \WH\cap \Linf$ turns out to be a weak solution of our original problem \eqref{problem2}.
	
	For the nonpositive solution we consider the functional $\Gamma^-\colon\WH\to \R$ and show in the same way that it has a global minimizer $v_0 \in \WH$ which belongs to $[\underline{v},0]$.
\end{proof}

Let us study now the case when the solutions depend on the first Robin eigenvalue. We set
\begin{align*}
&&h_1(x,s,\xi)&=(\zeta-\vartheta) |s|^{p-2}s-\mu(x)|s|^{q-2}s-f(x,s)&&\text{for a.\,a.\,}x\in\Omega,\\
&&h_2(x,s)&=-\beta|s|^{p-2}s&&\text{for a.\,a.\,}x\in\partial\Omega,
\end{align*}
for all $s\in\R$ with parameters $\zeta>\vartheta>0$ to be specified, $\beta>0$ is the same parameter as in the Robin eigenvalue problem and $f$ is a Carath\'eodory function. Then, problem \eqref{problem} becomes
\begin{equation}\label{problem3}
\begin{aligned}
\hspace*{-0.3cm}-\divergenz(|\nabla u|^{p-2} \nabla u+ \mu(x) |\nabla u|^{q-2} \nabla u)&= (\zeta-\vartheta) |u|^{p-2}u-\mu(x)|u|^{q-2}u- f(x,u) && \text{in } \Omega, \\
(|\nabla u|^{p-2} \nabla u+ \mu(x) |\nabla u|^{q-2} \nabla u) \cdot \nu&=- \beta|u|^{p-2}u&& \text{on } \rand,
\end{aligned}
\end{equation}
where $f$ satisfies the following assumptions:
\begin{enumerate}
	\item[\textnormal{(H3)}]
	The function $f\colon \Omega \times \R  \to \R $ is a Carath\'eodory function such that:
	\begin{enumerate}
		\item[\textnormal{(i)}]
		$f$ is bounded on bounded sets.
		\item[\textnormal{(ii)}]
		It holds
		\begin{align*}
		&\lim_{s \to \pm \infty}\,\frac{f(x,s)}{|s|^{q-2}s}=+\infty \quad\text{uniformly for a.\,a.\,}x\in\Omega.
		\end{align*}
		\item[\textnormal{(iii)}]
		It holds
		\begin{align*}
		&\lim_{s \to 0}\,\frac{f(x,s)}{|s|^{p-2}s}=0 \quad\text{uniformly for a.\,a.\,}x\in\Omega.
		\end{align*}
	\end{enumerate}
\end{enumerate}

We have the following multiplicity result concerning problem \eqref{problem3}.

\begin{theorem}\label{thm3}
	Let hypotheses  \eqref{condition_poincare} and \textnormal{(H3)} be satisfied. Further, let $\zeta>\lambda_{1,p,\beta}^{\text{R}}+\vartheta$ with $\vartheta >0$ and $\lambda_{1,p,\beta}^{\text{R}}$ being the first eigenvalue of the Robin eigenvalue problem of the $p$-Laplacian with $\beta>0$ stated in \eqref{robin}. Then, problem \eqref{problem3} has at least two nontrivial weak solutions $u_1, v_1\in \WH\cap \Linf$ such that $u_1\geq 0$ and $v_1\leq 0$.
\end{theorem} 

\begin{proof}
	Taking hypothesis \textnormal{(H3)(ii)} into account we find a constant $M=M(\zeta)>1$ such that
	\begin{align}\label{est_f_g2}
	\begin{split}
	f(x,s)s &\geq   \zeta |s|^q \quad \text{for a.a. } x \in \Omega \text{ and all }|s| \geq M.
	\end{split}
	\end{align}
	As in the proof of Theorem \ref{thm2}, by \eqref{est_f_g2}, we can take constant functions $\overline{u} \in  \left(M,+\infty\right)$ and $\underline{v}\equiv -\overline{u}$ such that
	\begin{align}\label{corollary_A_super2}
	0 \geq \zeta\overline{u}^{p-1}-f(x,\overline{u}) \quad \text{for a.\,a.\,}x\in\Omega
	\quad \text{and}\quad 
	0 \leq \zeta|\underline{v}|^{p-2}\underline{v}-f(x,\underline{v}) \quad \text{for a.\,a.\,}x\in\Omega,
	\end{align}
	because $p<q$ and $M>1$.
	
	Then we define truncations $\psi^{\pm}_\zeta\colon\Omega\times\R\to\R$ and $\psi^{\pm}_\beta\colon \partial\Omega\times\R\to\R$ as follows
	\begin{align}\label{truncation_one2}
	\begin{split}
	\psi^+_\zeta(x,s)&=
	\begin{cases}
	0 \qquad & \text{if } s<0\\
	\zeta s^{p-1}-f(x,s)& \text{if }0 \leq s \leq \overline{u}\\
	\zeta \overline{u}^{p-1}-f(x,\overline{u})& \text{if } \overline{u}<s
	\end{cases},\\
	\psi^+_\beta(x,s)&=
	\begin{cases}
	0 \qquad & \text{if } s<0\\
	-\beta s^{p-1} & \text{if }0 \leq s \leq \overline{u}\\
	-\beta \overline{u}^{p-1} & \text{if } \overline{u}<s
	\end{cases},\\
	\psi^-_\zeta(x,s)&=
	\begin{cases}
	\zeta|\underline{v}|^{p-2}\underline{v}-f(x,\underline{v}) & \text{if } s<\underline{v}\\
	\zeta|s|^{p-2}s-f(x,s) & \text{if }\underline{v} \leq s \leq 0\\
	0 & \text{if } 0<s
	\end{cases},\\
	\psi^-_\beta(x,s)&=
	\begin{cases}
	-\beta |\underline{v}|^{p-2}\underline{v} \qquad & \text{if } s<\underline{v}\\
	-\beta |s|^{p-2}s & \text{if }\underline{v} \leq s \leq 0.\\
	0 & \text{if } 0<s
	\end{cases}
	\end{split}
	\end{align}
	We set
	\begin{align*}
	\Psi^{\pm}_\zeta(x,s)=\int^s_0 \psi^{\pm}_\zeta(x,t)\diff t
	\quad\text{and}\quad
	\Psi^{\pm}_\beta(x,s)=\int^s_0 \psi^{\pm}_\beta(x,t)\diff t
	\end{align*}
	and introduce the $C^1$-functionals $\Pi^{\pm}\colon \WH \to \R$ given by
	\begin{align*}
	\Pi^{\pm}(u)
		&=\frac{1}{p}\|\nabla u\|_p^p+\frac{1}{q}\|\nabla u\|_{q,\mu}^q+\frac{\vartheta}{p}\| u\|_p^p+\frac{1}{q}\|u\|_{q,\mu}^q-\into \Psi^{\pm}_\zeta(x,u) \diff x-\int_{\partial \Omega}\Psi^{\pm}_\beta(x,u) \diff\sigma.
	\end{align*}
	As before, we define $F(x,s)=\int^s_0 f(x,t)\diff t$.
	
	We start with the existence of a nonnegative solution. Because of \eqref{truncation_one2} we know that  the functional $\Gamma^+$ is coercive and also sequentially weakly lower semicontinuous. Therefore, we find an element  $u_1\in\WH$ such that
	\begin{align*}
		\Pi^+(u_1)= \inf \left[\Pi^+(u)\,:\, u \in \WH \right].
	\end{align*}
	By hypothesis \textnormal{(H3)(iii)}, we find for every $\eps>0$ a number $\delta\in (0,\overline{u})$ such that 
	\begin{align}\label{estimate_FG2}
	\begin{split}
	& F(x,s) \leq \frac{\varepsilon}{p} |s|^p	 \quad \text{for a.a. }x \in \Omega \text{ and for all } |s|\leq \delta.
	\end{split}
	\end{align}
	We recall that $u_{1,p,\beta}^{\text{R}}$ is the first eigenfunction corresponding to the first eigenvalue $\lambda_{1,p,\beta}^{\text{R}}$ of the eigenvalue problem of the $p$-Laplacian with Robin boundary condition, see \eqref{robin}. Without any loss of generality we can assume that $u_{1,p,\beta}^{\text{R}}$ is normalized (that is, $\|u_{1,p,\beta}^{\text{R}}\|_{p}=1$) and because of $u_{1,p,\beta}^{\text{R}}\in \interior$ we choose $t \in (0,1)$ sufficiently small such that $t u_{1,p,\beta}^{\text{R}}(x) \in [0,\delta]$ for all $x \in \overline{\Omega}$. Applying  \eqref{truncation_one2}, \eqref{estimate_FG2}, $\delta<\overline{u}$ and $\vartheta >0$  gives
	\begin{align}\label{eq_5}
	\begin{split}
	&\Pi^+\l(t u_{1,p,\beta}^{\text{R}}\r)\\
	&=\frac{1}{p}\l\|\nabla \l(t u_{1,p,\beta}^{\text{R}}\r)\r\|_p^p+\frac{1}{q}\l\|\nabla \l(t u_{1,p,\beta}^{\text{R}}\r)\r\|_{q,\mu}^q+\frac{\vartheta}{p}\l\| t u_{1,p,\beta}^{\text{R}}\r\|_p^p+\frac{1}{q}\l\|t u_{1,p,\beta}^{\text{R}}\r\|_{q,\mu}^q\\
	&\quad -\into \Psi^{+}_\zeta\l(x,t u_{1,p,\beta}^{\text{R}}\r) \diff x-\int_{\partial \Omega}\Psi^{+}_\beta\l(x,t u_{1,p,\beta}^{\text{R}}\r) \diff\sigma\\
	&\leq \frac{t^p}{p} \lambda_{1,p,\beta}^{\text{R}}-\frac{\beta t^p}{p}\l\|u_{1,p,\beta}^{\text{R}}\r\|_{p,\partial\Omega}^p+\frac{t^q}{q}\l\|\nabla  u_{1,p,\beta}^{\text{R}}\r\|_{q,\mu}^q+\frac{t^p\vartheta}{p}+\frac{t^q}{q}\l\|u_{1,p,\beta}^{\text{R}}\r\|_{q,\mu}^q\\
	&\quad -\frac{\zeta t^p}{p} +\into F\l(x,tu_{1,p,\beta}^{\text{R}}\r)\diff x+\frac{\beta t^p}{p}\l\|u_{1,p,\beta}^{\text{R}}\r\|_{p,\partial\Omega}^p\\
	&\leq \frac{t^p}{p} \lambda_{1,p,\beta}^{\text{R}}+\frac{t^q}{q}\l\|\nabla  u_{1,p,\beta}^{\text{R}}\r\|_{q,\mu}^q+\frac{t^p\vartheta}{p}+\frac{t^q}{q}\l\|u_{1,p,\beta}^{\text{R}}\r\|_{q,\mu}^q -\frac{\zeta t^p}{p}+\frac{\eps t^p}{p}\\
	&\leq t^p \l(\frac{\lambda_{1,p,\beta}^{\text{R}}+\vartheta-\zeta+\eps}{p} \r)
	+t^q\l(\frac{\l\|\nabla  u_{1,p,\beta}^{\text{R}}\r\|_{q,\mu}^q+\l\|u_{1,p,\beta}^{\text{R}}\r\|_{q,\mu}^q}{q}\r).
	\end{split}
	\end{align}
	Due to $\zeta>\lambda_{1,p,\beta}^{\text{R}}+\vartheta$ and  $p<q$ one has from  \eqref{eq_5} for $\eps\in (0,\zeta-\lambda_{1,p,\beta}^{\text{R}}-\vartheta)$ that
	\begin{align*}
	\Pi^+\l(t u_{1,p,\beta}^{\text{R}}\r)<0 \quad\text{for all sufficiently small }t>0.
	\end{align*}
	Hence, $\Pi^+\l(u_1\r)<0=\Pi^+\l(0\r)$ and so $u_1\neq 0$.
	
	We have $(\Pi^+)'(u_1)=0$, that is,
	\begin{align}\label{17}
	\begin{split}
	& \into \left(|\nabla u_1|^{p-2}\nabla u_1+\mu(x)|\nabla u_1|^{q-2}\nabla u_1 \right)\cdot\nabla \ph \diff x\\
	&\quad +\into \left(\vartheta |u_1|^{p-2} u_1+\mu(x)|u_1|^{q-2} u_1 \right) \ph \diff x\\ 
	& =\into \psi^+_\zeta\l(x,u_1\r)\ph\diff x+\int_{\partial\Omega} \psi_\beta^+\l(x,u_1\r) \ph\diff \sigma
	\end{split}
	\end{align}
	for all $\ph \in \WH$. As done in the proof of Theorem \ref{thm2} we take $\ph=-u_1^-\in\WH$ and  $\ph=\l(u_1-\overline{u}\r)^+ \in \WH$ as test functions in \eqref{17} which gives us $0 \leq u_1 \leq \overline{u}$, see \eqref{corollary_A_super2}. Hence, by the definition of the truncations in \eqref{truncation_one2} we see that $u_1\in\WH\cap\Linf$ solves problem \eqref{problem3}. 
	
	In the same way we can show the existence of a nontrivial nonpositive solution $v_1\in\WH\cap \Linf$ by treating the functional $\Pi^-\colon \WH\to\R$ instead of $\Pi^+\colon\WH\to\R$.
\end{proof}

\begin{remark}
	In this section we decided to consider two different problems since in the proof of Theorem \ref{thm2} the use of the first Robin eigenfunction would have provided a condition of the form
	\begin{align}\label{condition_robin}
		\lambda_{1,p,\beta}^{\text{R}}+\vartheta< (\beta+\zeta) \l\|u_{1,p,\beta}^{\text{R}} \r\|_{p,\partial\Omega}^p,
	\end{align}
	which depends also on the boundary norm of the eigenfunction $u_{1,p,\beta}^{\text{R}}$. So the statement of Theorem \ref{thm2} still holds true when we replace the assumption $\zeta>\lambda_{1,p}^{\text{S}}$ by \eqref{condition_robin} where $u_{1,p,\beta}^{\text{R}}$ is the first normalized (that is,  $\|u_{1,p,\beta}^{\text{R}} \|_{p}=1$) eigenfunction associated to the first eigenvalue $	\lambda_{1,p,\beta}^{\text{R}}$ of the Robin eigenvalue problem. 
\end{remark}

%

\end{document}